\documentclass{amsart}
\usepackage{amsmath,amsfonts,euscript,amscd,amsthm,amssymb,upref,graphics}
\usepackage{mathrsfs}

\usepackage[all]{xy}
\usepackage{tikz-cd}

\usepackage{latexsym}
\usepackage{amsfonts,amssymb} 
\usepackage{graphicx}
\usepackage{tikz}
\usepackage{caption}
\usepackage{subcaption}
\usepackage{enumitem}

\newcommand{\field}[1]{\mathbb{#1}}
\newcommand{\A}{\field{A}}
\newcommand{\C}{\field{C}}

\newcommand{\G}{\field{G}}

\newcommand{\N}{\field{N}}
\newcommand{\PP}{\field{P}}
\newcommand{\Q}{\field{Q}}

\newcommand{\V}{\field{V}}
\newcommand{\Z}{\field{Z}}

\newcommand{\krn}{{\rm ker}\,}

\theoremstyle{plain}

\newtheorem{theorem}{Theorem}[section]
\newtheorem{proposition}[theorem]{Proposition}
\newtheorem{lemma}[theorem]{Lemma}
\newtheorem{corollary}[theorem]{Corollary}
\newtheorem{definition}[theorem]{Definition}

\newtheorem{question}[theorem]{Question}

\theoremstyle{definition}

\theoremstyle{remark}

\begin{document}

\makeatletter	   
\makeatother     

\title{a note on smooth affine $SL_2$-surfaces}
\author{G. Freudenburg}
\date{\today}

\begin{abstract} Working over a field $k$ of characteristic zero, we study the ring $\mathfrak{R}=\mathfrak{D}^{\Z_2}$ where 
$\mathfrak{D}=k[x_0,x_1,x_2]/(2x_0x_2-x_1^2-1)$ and $\Z_2$ acts by $x_i\to -x_i$. $\mathfrak{D}$ admits an algebraic $SL_2(k)$-action which restricts to $\mathfrak{R}$. 
Our results include the following. 
(1) If $k$ is algebraically closed, the smooth $SL_2$-surface $X={\rm Spec}(\mathfrak{R})$ admits an algebraic embedding in $\A_k^4$, and for any such embedding
the $SL_2(k)$-action on $X$ does not extend to $\A_k^4$. 
In addition, there is no algebraic embedding of $X$ in $\A_k^3$. 
(2) The automorphism group ${\rm Aut}_k(\mathfrak{R})$ acts transitively on the set of irreducible locally nilpotent derivations of $\mathfrak{R}$. 
(3) Every automorphism of $\mathfrak{R}$ extends to $\mathfrak{D}$, and ${\rm Aut}_k(\mathfrak{R})=PSL_2(k)\ast_H\mathcal{T}$ where $\mathcal{T}$ is its triangular subgroup.
(4) $\mathfrak{R}$ is non-cancellative, i.e., there exists a ring $\tilde{\mathfrak{R}}$ such that $\mathfrak{R}^{[1]}\cong_k\tilde{\mathfrak{R}}^{[1]}$ but
$\mathfrak{R}\not\cong_k\tilde{\mathfrak{R}}$. In order to distinguish $\mathfrak{R}$ from $\tilde{\mathfrak{R}}$, we calculate
the plinth invariant for $\mathfrak{R}$. 
\end{abstract}

\maketitle

\section{Introduction}
Over an algebraically closed field $k$ of characteristic zero, there are exactly three smooth affine $SL_2$-surfaces with trivial units over $k$, namely,
\begin{enumerate}
\item the plane $\A_k^2$
\item the quadric surface $Y={\rm Spec}(\mathfrak{D})$ where $\mathfrak{D}=k[x_0,x_1,x_2]$ and $2x_0x_2-x_1^2=1$, and
\item the quotient $X=Y/\Z_2={\rm Spec}\left( \mathfrak{D}^{\langle\mu\rangle}\right)$ where $\langle\mu\rangle\cong\Z_2$ acts on $\mathfrak{D}$ by $\mu (x_i)=-x_i$.
\end{enumerate}
This was shown by Popov in \cite{Popov.73a,Popov.73b}.
Both $X$ and $Y$ are homogeneous spaces for $SL_2(k)$, namely, 
$Y\cong_kSL_2(k)/T$ and $X\cong_kSL_2(k)/N$ where $T$ is the one-dimensional torus and $N$ is its normalizer; see \cite{Popov.92}, 6.1, Theorem 2. 
These surfaces can also be described by $Y=\PP^1\times\PP^1\setminus\Delta$ where $\Delta$ is the diagonal, and
$X=\PP^2\setminus Q$ where $Q$ is a quadric curve. 

The surfaces $\A_k^2$ and $Y$ have been studied extensively. The purpose of this note is to investigate $X$. 
Despite its importance, the literature about this surface is rather scant, 
apart from that on general families to which it belongs, e.g., Gizatullin surfaces or homogeneous spaces. 
A brief description of known results is given at the end of the {\it Introduction}. 
Our main results are laid out in the next five theorems, in which $k$ is any field of characteristic zero and $\bar{k}$ is its algebraic closure. 

Observe that the $SL_2(k)$-action on $\A_k^3$ induced by the irreducible $SL_2$-module of dimension $3$ restricts to $Y\subset\A_k^3$. 
In the following theorem, embeddings are understood to be closed embeddings. 

\begin{theorem}\label{non-extend}  Assume that $k=\bar{k}$. 
Let $X$ be equipped with the $SL_2(k)$-action defined by its structure as a homogeneous space for $SL_2(k)$, 
let $\A_k^5$ be equipped with the $SL_2(k)$-action induced by the irreducible $SL_2$-module of dimension $5$,  
and let $\A_k^4$ be equipped with any algebraic $SL_2(k)$-action. 
\begin{itemize}
\item [{\bf (a)}] There exists an equivariant algebraic embedding $X\to\A_k^5$.
\item [{\bf (b)}] There exists an algebraic embedding $X\to\A_k^4$.
\item [{\bf (c)}] There is no equivariant algebraic embedding $X\to\A_k^4$. 
\item [{\bf (d)}] There is no algebraic embedding $X\to\A_k^3$.
\item [{\bf (e)}] The cylinders $X\times\A_k^1$ and $Y\times\A_k^1$ are not isomorphic as algebraic $k$-varieties. 
\end{itemize}
\end{theorem}
The proof of part (c) relies on the theorem of Panyushev \cite{Panyushev.84}, which asserts that, when $k$ is algebraically closed, 
every algebraic action of $SL_2(k)$ on $\A_k^4$ is induced by an $SL_2$-module. 
The proofs of parts (d) and (e) rely on results of Bandman and Makar-Limanov \cite{Bandman.Makar-Limanov.01}.

For an affine $k$-domain $B$, ${\rm LND}(B)$ is the set of locally nilpotent derivations of $B$, and ${\rm Aut}_k(B)$ is the group of $k$-algebra automorphisms of $B$. 
Given $D\in {\rm LND}(B)$, $\krn D$ is its kernel and ${\rm pl}(D)=\krn D\cap DB$ is its plinth ideal. The Makar-Limanov invariant of $B$ or ${\rm Spec}(B)$ is:
\[
 ML(B)=\bigcap_{D\in{\rm LND}(B)}\krn D
 \] 
Let $\mathcal{C}(k)$ denote the class of rings $B$ which satisfy:
\begin{itemize}
\item [(a)] $B$ is a normal affine $k$-domain, $\dim B=2$, $B^*=k^*$ and $ML(B)\ne B$.
\item [(b)] $\krn D\cong_kk^{[1]}$ for every nonzero $D\in {\rm LND}(B)$. 
\end{itemize}
Conditions (a) and (b) imply condition
\begin{itemize}
\item [${\rm (b)}^{\prime}$] $B$ is rational
\end{itemize}
and if $k=\bar{k}$, then conditions (a) and ${\rm (b)}^{\prime}$ imply condition (b); see {\it Section\,\ref{last}}. 
Partition $\mathcal{C}(k)$ as $\mathcal{C}_0(k)\cup\mathcal{C}_1(k)$ where $ML(B)=k$ for $B\in\mathcal{C}_0(k)$ and $ML(B)=k^{[1]}$ for $B\in\mathcal{C}_1(k)$. 
The rings $k^{[2]}$, $\mathfrak{D}$ and $\mathfrak{D}^{\langle\mu\rangle}$ belong to $\mathcal{C}_0(k)$. 

For $B\in\mathcal{C}(k)$ other than $k^{[2]}$, the isomorphism class of $\krn D/{\rm pl}(D)$ as a $k$-algebra is an invariant of the conjugacy class of $D\in {\rm LND}(B)$. 
Since ${\rm pl}(D)\notin\{ (0),(1)\}$, this is a zero-dimensional ring, and when $k=\bar{k}$, ${\rm Spec}(\krn D/{\rm pl}(D))$ is a weighted configuration of points in $\A_k^1$. 
We make the following definitions.
\begin{itemize}
\item ${\rm KLND}(B)=\{ \krn D\, |\, D\in{\rm LND}(B)\, ,\, D\ne 0\}$
\item ${\rm ILND}(B)=\{ D\in{\rm LND}(B)\, |\, D\,\, \text{is\, irreducible}\}$
\item $B$ has the {\bf transitivity property} if the action of ${\rm Aut}_k(B)$ on ${\rm KLND}(B)$ is transitive.
\item $B$ has the {\bf strong transitivity property} if the action of ${\rm Aut}_k(B)$ on ${\rm ILND}(B)$ is transitive.
\item Assume that $B$ has the strong transitivity property and $B\not\cong_kk^{[2]}$. The {\bf plinth invariant} of $B$ is the isomorphism class of $\krn D/{\rm pl}(D)$, $D\in{\rm ILND}(B)$.
\end{itemize}
{\it Lemma\,\ref{TP-STP}} shows that, when $k=\bar{k}$ and $B$ admits a nontrivial $\Z$-grading, the transitivity property implies the strong transitivity property. 
Note that every $B\in\mathcal{C}_1(k)$ satisfies the transitivity property, since there is only one element in ${\rm KLND}(B)$.

Ebey \cite{Ebey.62} showed that $k^{[2]}$ has the strong transitivity property. 
Daigle \cite{Daigle.03a} showed that $\mathfrak{D}$ has the transitivity property, and from this it is easy to show that $\mathfrak{D}$ has the strong transitivity property
and the plinth invariant of $\mathfrak{D}$ is $k$. 

Hereafter, we let $\mathfrak{R}=\mathfrak{D}^{\langle\mu\rangle}$.  

\begin{theorem}\label{transitive}
$\mathfrak{R}$ has the strong transitivity property and its plinth invariant is $k$. 
\end{theorem}
Daigle and Russell \cite{Daigle.Russell.04} showed that, if $k=\bar{k}$ and $B\in\mathcal{C}_0(k)$ is such that the smooth part 
of ${\rm Spec}(B)$ has finite Picard group, then the number of orbits in ${\rm KLND}(B)$ is either one or two. 
Moreover, $B$ has the transitivity property if and only if $B$ is symmetric at infinity (see the article for this definition). 

Danilov and Gizatullin \cite{Danilov.Gizatullin.77} showed that ${\rm Aut}_k(Y)\cong {\bf L}\ast_C{\bf T}$ where {\bf L} and {\bf T} are the 
linear and triangular automorphisms of $\A_k^3$ restricting to $Y$, 
and $C={\bf L}\cap{\bf T}$. Consequently, every automorphism of $Y$ is the restriction of an automorphism of $\A_k^3$.
{\bf L} can be viewed as the subgroup of $GL_3(k)$ preserving the quadratic form $2x_0x_2-x_1^2$, so ${\bf L}\cong O_3(k)$. 
Being concerned here with $SL_2(k)$-actions, we recognize that $SO_3(k)\cong PSL_2(k)$ and 
$O_3(k)\cong PSL_2(k)\times\langle\mu\rangle$. 
Observe that $\mu$ restricts to $\mathfrak{R}$, and that $\mu\vert_{\mathfrak{R}}=id_{\mathfrak{R}}$. 

Define $\delta'\in {\rm LND}(\mathfrak{D})$ by $\delta'=x_0\partial_{x_1}+x_1\partial_{x_2}$ and let $\delta\in {\rm LND}(\mathfrak{R})$ be the restriction of $\delta'$ to $\mathfrak{R}$. 
The action of $SL_2(k)$ on $\mathfrak{R}$ is given by a group homomorphism $SL_2(k)\to {\rm Aut}_k(\mathfrak{R})$ whose image is $PSL_2(k)$. 
Define the following subgroups of ${\rm Aut}_k(\mathfrak{R})$. 
\begin{enumerate}[label=(\roman*)]
\item $T\cong k^*$ is the torus in $PSL_2(k)$. 
\smallskip
\item $E=\{ \exp(f\delta )\,\vert\, f\in k[x_0^2]\}\cong (k^{[1]},+)$
\smallskip
\item $\mathcal{T}=\langle E,T\rangle\cong E\rtimes k^*$
\smallskip
\item $E_L=\{ \exp (t\delta )\,\vert\, t\in k\} \cong \G_a$
\smallskip
\item $H=PSL_(k)\cap\mathcal{T}=\langle E_L,T\rangle \cong \G_a\rtimes\G_m$
\end{enumerate}

 \begin{theorem}\label{aut-group}  With the notation and hypotheses above:
 \begin{itemize}
 \item [{\bf (a)}] ${\rm Aut}_k(\mathfrak{R})$ is the subgroup of ${\rm Aut}_k(\mathfrak{D})$ consisting of elements which restrict to $\mathfrak{R}$. 
 \item [{\bf (b)}] ${\rm Aut}_k(\mathfrak{R})\cong PSL_2(k)\ast_H\mathcal{T}$
 \end{itemize}
\end{theorem}
 
 Define $X_n={\rm Spec}(B_n)$ for $B_n=k[x,y,z]/(x^nz-y^2-1)$, $n\ge 0$. 
It is well-known that $X_m\not\cong_kX_n$ if $m\ne n$, whereas $X_m\times\A_k^1\cong_kX_n\times\A_k^1$ for all $m,n\ge 1$. 
To distinguish $X_1$ from $X_2$, Danielewski \cite{Danielewski.89} calculated the fundamental group at infinity for $X_1$ and $X_2$ ($k=\C$), and
Makar-Limanov \cite{Makar-Limanov.01a} calculated $ML(X_1)=k$ and $ML(X_2)=k[x]$.
We can also distinguish $X_1$ from $X_2$ by the plinth invariant of $X_1$, i.e., we do not need to calculate any invariant of $X_2$. 
Rather, we have only to observe that $B_2$ has an irreducible locally nilpotent derivation with plinth ideal $(x^2)$ in $k[x]$. 
This is the idea used to prove part (b) of the next theorem.

 \begin{theorem}\label{non-cancel}
 Write $\mathfrak{R}=k[x_0^2,x_0x_1,x_1x_2,x_2^2]$. Let $\mathfrak{R}[V]\cong \mathfrak{R}^{[1]}$ and define 
 \[
 F=15x_0^2V-3(x_0x_1)(x_2^2)-2(x_1x_2) \quad \text{and}\quad \tilde{\mathfrak{R}}=\mathfrak{R}[V]/(F)\, .
 \]
 Let $\tilde{X}={\rm Spec}(\tilde{\mathfrak{R}})$. 
 \begin{itemize}
 \item [{\bf (a)}] $X\times\A_k^1\cong_k\tilde{X}\times\A_k^1$
 \smallskip
 \item [{\bf (b)}] $X\not\cong_k\tilde{X}$
 \smallskip
 \item [{\bf (c)}] $\tilde{X}$ admits an effective $\G_a\rtimes\G_m$-action where the $\G_a$-action is fixed-point free. 
 \smallskip
 \item [{\bf (d)}] If $k=\bar{k}$ then $\tilde{\mathfrak{R}}\in\mathcal{C}(k)$ and ${\rm Cl}(\tilde{\mathfrak{R}})=\Z_2$. 
 \end{itemize}
 \end{theorem}
 This example is constructed using the fundamental pair of derivations of $\mathfrak{R}$ which define the $SL_2(k)$-action. 
 
 Finally, the full class of surfaces $X_m$, $m\ge 1$, over $k=\C$ were distinguished from each other by Fieseler \cite{Fieseler.94} who calculated the fundamental group at infinity for each, thus showing that they are 
 pairwise non-homeomorphic. However, when $m\ge 2$, $ML(X_m)=k[x]$ (see \cite{Makar-Limanov.01a}, Proposition), so this invariant does not distinguish among $X_m$, $m\ge 2$. 
 
 \begin{theorem}\label{plinth-invariant} 
 For each $n\in\N$, $n\ge 1$, $B_n\in\mathcal{C}(k)$, $B_n$ has the strong transitivity property, and the plinth invariant of $B_n$ equals $k[x]/(x^n)$. 
 Consequently, $X_m\not\cong_kX_n$ if $m,n\ge 2$ and $m\ne n$. 
 \end{theorem}
 
 \subsection{Automorphism groups} Generators for $GA_2(k):={\rm Aut}_k(\A_k^2)$ were first given by Jung (1942) in \cite{Jung.42} in case the field $k$ is algebraically closed of characteristic zero. 
Jung showed that $GA_2(k)$ is generated by its affine linear and triangular subgroups, ${\rm Aff}_2(k)$ and $BA_2(k)$. 
A structural description of $GA_2(k)$  was subsequently given by 
Van der Kulk (1953) in \cite{Kulk.53} for any characteristic zero field. In particular:
\[
GA_2(k)={\rm Aff}_2(k)*_BBA_2(k)\quad {\rm where}\quad B={\rm Aff}_2(k)\cap BA_2(k)
\]
This was generalized to any field by Nagata (1972) in \cite{Nagata.72}. 

The exponential automorphisms of $\A_k^2$ were determined by Ebey (1962) in \cite{Ebey.62} and Rentschler (1968) in \cite{Rentschler.68}. 
Without using the results of Jung or Van der Kulk, they show that the tame subgroup of $GA_2(k)$ (generated by its linear and triangular elements) acts transitively 
on ${\rm KLND}(B)$. 
From this, Rentschler recovers Jung's theorem. Transitivity also follows from the amalgamated free product structure of $GA_2(k)$, since reductive subgroups must conjugate to $GL_2(k)$ and unipotent subgroups must conjugate to $BA_2(k)$; see \cite{Serre.80}, Theorem 8. 

For the affine surfaces $xy+az^2=b$ over a field $K$ ($a,b\in K^*$), Danilov and Gizatullin (1977) computed the automorphism groups as amalgamated free products in \cite{Danilov.Gizatullin.77}. 
Lamy (2005) derived the same structure for these groups by different methods \cite{Lamy.05} .
In the more general case, generators for the automorphism group of a Danielewski surface $S$ defined by $xy=P(z)$ were given by Makar-Limanov (1990) in \cite{Makar-Limanov.90}. 
Makar-Limanov showed that this group is generated by the linear and triangular automorphisms of $\A_k^3$ which restrict to $S$.
Let $d =P'(z)\partial_y-x\partial_z$ and, for $f(x)\in k[x]$, define $\Delta_f=\exp (fd )$. 
Without using Makar-Limanov's result, Daigle (2003) showed in \cite{Daigle.03a} that the subgroup generated by $\{ \Delta_f\, |\, f\in k[x]\}$ and the involution $\tau (x,y,z)=(y,x,z)$ acts transitively 
on ${\rm KLND}(k[S])$. 
Makar-Limanov gives another proof of Daigle's result in \cite{Makar-Limanov.03}. 
In \cite{Blanc.Bot.Poloni.23}, Blanc, Bot and Poloni (2023) used Makar-Limanov's generators of ${\rm Aut}_k(S)$ to show that ${\rm Aut}_k(S)$ is the amalgamated free product of its affine and triangular subgroups, 
i.e., elements that descend from an affine or triangular automorphism of $\A_k^3$, respectively. 

In \cite{Liendo.Regeta.Urech.23}, Liendo, Regeta and Urech (2023) prove that any normal affine $SL_2$-surface is uniquely determined by its automorphism group (Proposition 6.8). 

\subsection{The Surface $X$.} 
The papers \cite{Gizatullin.71c,Flenner.Zaidenberg.05b} describe normal affine surfaces over $\C$ which admit an action of an algebraic group with a big open orbit, i.e., the complement of the orbit is finite. 
In the first paper, Gizatullin (1971) deals with the smooth case, and shows that the only such surfaces are $\C\times\C$, $\C\times\C^*$, $\C^*\times\C^*$, $Y$ and $X$. 
In the second paper, Flenner and Zaidenberg (2005) generalize these results to the normal case. 

Working over $k=\C$, $\G_a$-actions on $Y$ and $X$ were studied by Fauntleroy and Magid (1978) in \cite{Fauntleroy.Magid.78}. 
Each surface has a fixed-point free $\G_a$-action which is the restriction of the $SL_2(\C )$-action to its upper-triangular unipotent subgroup. 
The authors show that each action has a geometric quotient, where $Y/\G_a\cong L$ is not affine, and $X/\G_a\cong\A_k^1$ is affine. 
In particular, $L$ denotes a line with the origin doubled. The $\G_a$-action on $Y$ is locally trivial but the quotient is not separated, so the action is not proper.
On the other hand, the action on $X$ is not locally trivial -- otherwise it would be globally trivial since $X$ would be an $\A_k^1$-bundle over $\A_k^1$, and this is clearly not the case. 
Since a proper action is locally trivial, this implies that the $\G_a$-action on $X$ is not proper. 
The authors also calculate the divisor class group (which equals the Picard group) for each surface: ${\rm Cl}(Y)\cong \Z$ and ${\rm Cl}(X)\cong\Z_2$. 

These surfaces also appear in Huckleberry's study (1986) of more general complex homogeneous surfaces \cite{Huckleberry.86}. 
In \cite{Masuda.Miyanishi.03}, Masuda and Miyanishi (2003) characterize $X$ as the unique $\Q$-homology plane $S$ with $ML(S)=\C$ and $H_1(S,\Z )=\pi_1(S)=\Z_2$. Recall that
a $\Q$-homology plane $S$ has $H_i(S,\Q )=0$ for $i\ge 1$. 
$Y$ and $X$ also feature in \cite{Flenner.Zaidenberg.05a}, in which Flenner and Zaidenberg (2005) study surfaces that admit an effective action of a maximal torus $\mathbb{T}$ such that any $\C^*$-action is conjugate to a subtorus of $\mathbb{T}$. 

Bandman and Makar-Limanov (2005) consider the surface $X$ over $k=\C$ in \cite{Bandman.Makar-Limanov.05}, Example 1. Of particular note, they give a simple set of defining relations for 
$X$ as a subvariety of $\A_{\C}^5$, and they give two independent locally nilpotent derivations on $\mathfrak{R}$. We give these relations in {\it Section\,\ref{XandY}}. 
Since the relations are polynomials over $\Z$, they define $X$ over any characteristic zero field.

In \cite{Freudenburg.24}, the author (2024) generalizes Popov's classification of normal affine $SL_2$-surfaces with trivial units to any field $k$ of characteristic zero. 
Proposition 5.3 shows that $X$ is not isomorphic to a hypersurface of the form $xy=p(z)$ for $p(z)\in k[z]$ (2024). 
This was shown independently by Arzhantsev and Zaitseva (2024) in \cite{Arzhantsev.Zaitseva.24} for $k=\C$.

\subsection{Conventions} Throughout, $k$ is a field of characteristic zero and $\bar{k}$ is its algebraic closure. 
A $k$-domain $B$ is an integral domain containing $k$. Given $b\in B$, $B_b=B[b^{-1}]$.
For the $k$-domain $B$, $B^{[n]}$ is the polynomial ring in $n$ variables over $B$ and ${\rm frac}(B)$ is its field of fractions. 
Given $k$-domains $A$ and $B$, $A\cong_kB$ indicates that $A$ and $B$ are isomorphic as $k$-algebras. 
Similarly, if $U$ and $V$ are algebraic $k$-varieties, $U\cong_kV$ indicates that $U$ and $V$ are isomorphic as $k$-varieties. 
$\G_a$ denotes the additive group $k$, and $\G_m$ the multiplicative group $k^*$.


\section{locally nilpotent derivations and fundamental pairs}\label{prelim}

\subsection{Locally nilpotent derivations} Let $B$ be an affine $k$-domain. The set of $k$-derivations of $B$ is ${\rm Der}_k(B)$. 
Given $D\in {\rm Der}_k(B)$, the kernel of $D$ is denoted $\krn D$, and $D$ is {\bf locally nilpotent} if, given $b\in B$, $D^nb=0$ for $n\gg 0$. 
The set of locally nilpotent derivations of $B$ is ${\rm LND}(B)$. This set is in bijective correspondence with the set of $\G_a$-actions on ${\rm Spec}(B)$ via the exponential map. 
We say that $D\in {\rm LND}(B)$ is {\bf irreducible} if the image $DB$ is contained in no proper principal ideal of $B$. 

Given $D\in {\rm LND}(B)$, set $A=\krn D$. An element $s\in B$ is a {\bf local slice} for $D$ if $Ds\ne 0$ and $D^2s=0$. 
A local slice $s$ for $D$ is a {\bf slice} if $Ds=1$. The {\bf Slice Theorem} says that, for any local slice $s$, $B_{Ds}=A_{Ds}[s]\cong_kA_{Ds}^{[1]}$. 

The following theorem is due to Vasconcelos. 
\begin{theorem}\label{Vasconcelos}  {\rm \cite{Vasconcelos.69}} Let $B$ be a $k$-domain with subalgebra $R$ such that $B$ is an integral extension of $R$. 
If $D\in {\rm Der}_k(B)$ restricts to a locally nilpotent derivation of $R$ then $D\in {\rm LND}(B)$.
\end{theorem}

See \cite{Freudenburg.17} for further details on locally nilpotent derivations. 

\subsection{Fundamental pairs} Suppose that $B$ is an affine $k$-domain. We can view ${\rm Der}_k(B)$ as a Lie algebra with the usual bracket operation. 
The pair $(D,U)\in {\rm LND}(B)^2$ is a {\bf fundamental pair} for $B$ if $D$ and $U$ satisfy the relations
\[
[D,[D,U]]=-2D \quad\text{and}\quad [U,[D,U]]=2U
\]
i.e., $D$ and $U$ generate a Lie subalgebra isomorphic to $\mathfrak{sl}_2(k)$. Fundamental pairs for $B$ are in bijective correspondence with algebraic actions of $SL_2(k)$ on $B$. 
In particular, an $SL_2(k)$-action on $B$ induces a fundamental pair $(D,U)$ for $B$, since the upper and lower unipotent triangular subgroups of $SL_2(k)$, 
each being isomorphic to $\G_a$, give rise to locally nilpotent derivations of $B$ which satisfy these relations. 

Suppose that $(D,U)$ is a fundamental pair for $B$. Set $E=[D,U]$. 
By \cite{Andrist.Draisma.Freudenburg.Huang.Kutzschebauch.ppt}, Theorem 2.6, $E$ is a semi-simple derivation which induces a $\Z$-grading of $B$ given by 
$B=\bigoplus_{d\in\Z}B_d$ where $B_d=\krn (E-dI)$. The kernel $A=\krn D$ is a graded subring and $A=\bigoplus_{d\in\N}A_d$ is $\N$-graded.   
In addition:
\[
A_0=\krn D\cap\krn U=B^G \quad \text{where}\quad G=SL_2(k)
\]
Each orbit $G\cdot f$ of $f\in B$ is a finite dimensional vector space, and we denote this orbit by $\widehat{f}$. 
If $f\in A_d$ then a basis for $\widehat{f}$ is $\{ U^jf\, |\, 0\le j\le d\}$. 

Suppose that $B'$ is another $k$-domain endowed with the fundamental pair $(D',U')$. 
A $k$-algebra morphism $\varphi :B'\to B$ is {\bf equivariant} if $D\varphi =\varphi D'$ and $U\varphi =\varphi U$. 

Given $n\in\N$, let $\V_n(k)=kX_0\oplus\cdots\oplus kX_n$ be the irreducible $SL_2$-module of dimension $n+1$ over $k$. 
Then $\V_n(k)$ induces the fundamental pair $(D_n,U_n)$ for the polynomial ring $k[\V_n(k)]=k[X_0,\hdots X_n]$ defined by
\[
D_n(X_i)=X_{i-1} \,\, (1\le i\le n) \,\, ,\,\, D_n(X_0)=0
\]
and:
\[
U_n(X_i)=(i+1)(n-i)X_{i+1}\,\, (0\le i\le n-1) \,\, ,\,\, U(X_n)=0
\]
\begin{lemma}\label{cables} {\rm \cite{Andrist.Draisma.Freudenburg.Huang.Kutzschebauch.ppt}} Let $f_1,\cdots ,f_r\in A$ be nonzero and homogeneous such that $A=k[f_1,\hdots ,f_r]$.
\begin{itemize}
\item [{\bf (a)}] $B=k[\widehat{f_1},\hdots ,\widehat{f_r}]$
\item [{\bf (b)}]  Let $d_i\in\N$ be such that $f_i\in A_{d_i}$, $1\le i\le r$. 
There exists an equivariant surjection:
\[
 \varphi :k[\V_{d_1}(k)\oplus\cdots\oplus\V_{d_r}(k)]\to B
 \]
\end{itemize} 
\end{lemma} 

See \cite{Freudenburg.24, Andrist.Draisma.Freudenburg.Huang.Kutzschebauch.ppt} for further details on fundamental pairs. 


\section{The surfaces $Y$ and $X$}\label{XandY}

\subsection{Two homogeneous spaces for $SL_2(k)$}\label{hom-space}
In this section, assume that $k=\bar{k}$.
Define $\xi\in SL_2(k)$  and the one-dimensional torus $T\subset SL_2(k)$ by:
\[
\xi =\begin{pmatrix} 0&-1\cr 1&0\end{pmatrix} \quad\text{and}\quad T=\begin{pmatrix}\lambda & 0\cr 0&\lambda^{-1}\end{pmatrix}_{\lambda\in k^*}
\]
Define $K=\langle\xi\rangle\cong\Z_4$ and $N=TK$, the normalizer of the torus in $SL_2(k)$. 
Then $T$ and $N$ are the only closed one-dimensional reductive subgroups of $SL_2(k)$. 
By Matsushima's Criterion, the homogeneous spaces $SL_2(k)/T$ and $SL_2(k)/N$ are affine surfaces. We show directly that:
\[
SL_2(k)/T\cong_k Y \quad\text{and}\quad SL_2(k)/N\cong_k X
\]
Let $\begin{pmatrix} x&y\cr z&w\end{pmatrix}$ be coordinates for $SL_2(k)$, where $xw-yz=1$, and let $B=k[x,y,z,w]$ be its coordinate ring. 
The actions of $T$ and $H$ on $B$ are given by:
\[
\lambda\cdot (x,y,z,w)=(\lambda x,\lambda y,\lambda^{-1}z,\lambda^{-1}w) \quad\text{and}\quad \xi\cdot (x,y,z,w)=(-z,-w,x,y)
\]
Therefore, $B^T=k [xw,xz,yz,yw]$. Set $a=xw, b=xz, c=yz, d=yw$. Then $a=c+1$ and $bd=c(c+1)$. We thus see that $B^T\cong_k\mathfrak{D}$. 

Let $\tilde{c}=c+\frac{1}{2}$ so that $B^T=k[b,\tilde{c},d]$. We have $B^N=(B^T)^K=k[b,\tilde{c},d]^K$. Since
\[
\xi\cdot (b,\tilde{c},d)=(-b,-\tilde{c},-d)
\]
we obtain $B^N=(B^T)^{K/\Z_2}\cong_k\mathfrak{D}^{\langle\mu\rangle}=R$. 

\subsection{Fibers for $Y$ and $X$}
Let $\G_a$ act on $Y$ by $\exp (t\delta^{\prime})$, and on $X$ by $\exp (t\delta )$, $t\in k$. 
Let $\pi_1:Y\to \A_k^1$ and $\pi_2:X\to\A_k^1$ be their respective algebraic quotient maps, i.e., induced by inclusions:
\[
\krn\delta^{\prime}=k[x_0]\subset\mathfrak{D}\quad\text{and}\quad \krn\delta=k[x_0^2]\subset\mathfrak{R}
\]
\begin{lemma}\label{fibers} Let $Z=\A_k^1\setminus\{ 0\}$.
 \begin{itemize}
 \item [{\bf (a)}] The $\G_a$-actions on $Y$ and $X$ are fixed-point free. 
 \item [{\bf (b)}] $\pi_1^{-1}(Z)\cong_kZ\times\A_k^1$ and $\pi_2^{-1}(Z)\cong_kZ\times\A_k^1$ and these isomorphisms are $\G_a$-equivariant. 
 \item [{\bf (c)}] The fiber $\pi_1^{-1}(0)$ is reduced, with coordinate ring $k[x_1]/(x_1^2+1)^{[1]}$. 
  \item [{\bf (d)}] The fiber $\pi_2^{-1}(0)$ is non-reduced, with coordinate ring $\kappa^{[1]}$, where $\kappa\cong_kk[t]/(t^2)$. 
 \end{itemize}
\end{lemma}

\begin{proof} 
Part (a). We have:
\[
\mathfrak{D}=k[x_0,x_1,x_2]/(2x_0x_2-x_1^2-1) \quad\text{and}\quad 
\mathfrak{R}=k[x_2^2, x_1x_2,x_0x_2,x_1^2,x_0x_1,x_0^2]
\]
 For $Y$, the ideal $I_Y$ defining the fixed points is the ideal generated by the image of $\delta^{\prime}$.
 Therefore, $I_Y=(x_0,x_1)=(1)$ and the $\G_a$-action on $Y$ is fixed-point free. 
 
 Likewise, for $X$, the ideal $I_X$ defining the fixed points is the ideal generated by the image of $\delta$.
 The action of $\delta$ on $\mathfrak{R}$ is given by:
\begin{equation}\label{delta}
x_2^2\to 2x_1x_2\to 2(x_0x_2+x_1^2)\to 6x_0x_1\to 6x_0^2
\end{equation}
Moreover, $x_0x_2+x_1^2=3x_0x_2-1=3(2(x_0^2)(x_2^2)-(x_0x_1)(x_1x_2))-1$.
Therefore, 
\[
I_X=(x_1x_2,x_0x_2+x_1^2,x_0x_1,x_0^2)=(1)
\]
and the $\G_a$-action on $X$ is fixed-point free. 
\medskip

Part (b). Since $x_0$ is a local slice for $\delta^{\prime}$ we have $\mathfrak{D}_{x_0}=k[x_0]_{x_0}^{[1]}$ by the Slice Theorem. 
Similarly, since $x_0^2$ is a local slice for $\delta$ we have $\mathfrak{R}_{x_0^2}=k[x_0^2]_{x_0^2}^{[1]}$.
\medskip

Part (c). The fiber $\pi_1^{-1}(0)$ is defined by the ideal $(x_0)$ in $\mathfrak{D}$. We see that:
\[
\mathfrak{D}/(x_0)\cong k[x_1]/(x_1^2+1)[x_2]\cong_kk[x_1]/(x_1^2+1)^{[1]}
\]

Part (d). The fiber $\pi_2^{-1}(0)$ is defined by the ideal $(x_0^2)$ in $\mathfrak{R}$. Define:
\[
x=x_0^2\,\, ,\,\, y=x_0x_1\,\, ,\,\, z=x_1x_2\,\, ,\,\, w=x_2^2\,\, ,\,\, p=2xw-yz
\]
Then $\mathfrak{R}=k[x,y,z,w]$. Note the following relations in $\mathfrak{R}$:
\[
y^2=x(2p-1)\,\, ,\,\, z^2=w(2p-1)\,\, ,\,\, yw=zp
\]
Set $\bar{\mathfrak{R}}=\mathfrak{R}/(x)=k[\bar{y},\bar{z},\bar{w}]$. The induced relations are:
\[
\bar{y}^2=0\,\, ,\,\, \bar{z}^2=-\bar{w}(2\bar{y}\bar{z}+1)\,\, ,\,\, \bar{y}\bar{w}=-\bar{y}\bar{z}^2
\]
Set $\kappa = k[\bar{y}]$ and write $\bar{\mathfrak{R}}=\kappa [\bar{z},\bar{w}]$. We have:
\[
\bar{z}^2=-2(\bar{y}\bar{w})\bar{z}-\bar{w}=2\bar{y}\bar{z}^3-\bar{w} \implies \bar{w}=2\bar{y}\bar{z}^3-\bar{z}^2
\]
Therefore, $\bar{\mathfrak{R}}=\kappa [\bar{z}]\cong_k\kappa^{[1]}$. 
\end{proof}

\subsection{Relations for $X$}\label{4-gen}

In Example 1 of \cite{Bandman.Makar-Limanov.05}, Bandman and Makar-Limanov realize $X$ as a subvariety of $\A_k^5$ as follows. Let $k[t,x,y,z,u]=k^{[5]}$. Then $X$ is defined by the prime ideal $J$ generated by the following six polynomials:
\[
y^2-x(z-1)\,\, ,\,\, yt-z(z-1)\,\, ,\,\, t^2-u(z-1) \,\, ,\,\, yu-tz\,\, ,\,\, yz-xt \,\,,\,\, xu-z^2
\]
Since the variable $z+(yt-xu)=(yt-z(z-1))-(xu-z^2)$ belongs to $J$, we see that $X$ is also a subvariety of $\A_k^4$. 
Let $P=xu-yt$. Then the ideal in $k[t,x,y,u]$ defining $X$ is generated by:
\[
y^2-x(P-1)\,\, ,\,\, yt-P(P-1)\,\, ,\,\, t^2-u(P-1)\,\, ,\,\, yu-tP\,\, ,\,\, tx-yP
\]
\subsection{Fundamental pairs for $\mathfrak{D}$ and $\mathfrak{R}$}
We have $k[\V_2(k)]=k[X_0,X_1,X_2]\cong_kk^{[3]}$ with fundamental pair:
\[
\textstyle D_2=X_0\frac{\partial}{\partial X_1}+X_1\frac{\partial}{\partial X_2}\quad\text{and}\quad 
U_2=2X_2\frac{\partial}{\partial X_1}+2X_1\frac{\partial}{\partial X_0}
\]
The ring of invariants is $k[2X_0X_2-X_1^2]$. Let $(\delta^{\prime},\upsilon^{\prime})$ be the induced fundamental pair on:
\[ 
\mathfrak{D}=k[x_0,x_1,x_2]=k[X_0,X_1,X_2]/(2X_0X_2-X_1^2-1)
\]
The action of $\langle\mu\rangle\cong\Z_2$ on $\mathfrak{D}$ commutes with $(\delta^{\prime},\upsilon^{\prime})$, meaning that $(\delta^{\prime},\upsilon^{\prime})$ restricts to a fundamental pair
$(\delta ,\upsilon )$ on:
\[ 
\mathfrak{R}=\mathfrak{D}^{\langle\mu\rangle}=k[x_0^2,x_0x_1,x_0x_2,x_1^2,x_1x_2,x_2^2]
\]

\subsection{Three theorems} 
Given $f\in\krn\delta' =k[x_0]$, $\Delta_f\in{\rm Aut}_k(\mathfrak{D})$ is $\Delta_f=\exp (f\delta')$. 
Define $\tau\in {\rm Aut}_k(\mathfrak{D})$ by $\tau =(x_2,x_1,x_0)$
and let $G=\langle \tau ,\Delta_f\, |\, f\in k[x_0]\rangle$. 
The following is Daigle's Transitivity Theorem  applied to the ring $\mathfrak{D}$. 
\begin{theorem}\label{Daigle} 
$G$ acts transitively on ${\rm KLND}(\mathfrak{D})$.
\end{theorem}

Wright gives the following in \cite{Wright.78}, Proposition 2.
\begin{theorem}\label{Wright} Let $G$ be a group with subgroups $A,B\subset G$ such that $G\cong A\ast_CB$, where $C=A\cap B$. 
Let $H\subseteq G$ be a subgroup and write $A'=A\cap H$ and $B'=C\cap H$. The following are equivalent.
\begin{enumerate}[label=(\roman*)]
\item $H$ is generated by $A'$ and $B'$.
\item $H=A'\ast_{C'}B'$ where $C'=A'\cap B'$.
\end{enumerate}
 \end{theorem}

In \cite{Bandman.Makar-Limanov.01}, Bandman and Makar-Limanov study the class of all smooth affine surfaces $S$ over $\C$ with $ML(S)=\C$.
They show the following.
\begin{theorem}\label{Band-ML} Let $S$ be a smooth affine surfaces over $\C$ with $ML(S)=\C$. 
The following conditions are equivalent.
\begin{enumerate}[label=(\roman*)]
\item $S$ is isomorphic to a hypersurface in $\A_{\C}^3$. 
\item $S$ is isomorphic to $\{ (x,y,t)\in\A_k^3\, |\, xy=p(t)\}$ for some $p\in k[t]$ with simple roots.
\item $S$ admits a fixed-point free $\G_a$-action with reduced fibers.
\end{enumerate}
\end{theorem}
The equivalence of conditions (i) and (iii) is gotten by combining Theorem 1 and Theorem 2. The equivalence of conditions (i) and (ii) is found in the proof of Lemma 5. 
Though stated for $k=\C$, this result holds more generally for any algebraically closed field $k$ of characteristic zero. 


\section{Proof of {\it Theorem\,\ref{non-extend}}}\label{non-extend-proof}

{\bf Part (a).} 
Let $\mathfrak{R}=\bigoplus_{d\in\Z}\mathfrak{R}_d$ be the $\Z$-grading induced by $(\delta ,\upsilon )$, and let $A=\krn\delta$ with $\N$-grading $A=\bigoplus_{d\in\N}A_d$. 
Since $2x_0x_2-x_1^2=1$ and $\det \begin{pmatrix} 1&1\cr 2&-1\end{pmatrix}=-3$, we see from line (\ref{delta}) that:
\[
\mathfrak{R}=k[x_0^2, x_0x_1,x_0x_2+x_1^2,x_1x_2,x_2^2]=k[\widehat{x_0^2}]
\]
Since $\deg (x_0^2)=4$ there is an equivariant surjection $\varphi :k[\V_4(k)]\to \mathfrak{R}$. This proves part (a). 

\medskip

{\bf Part (b).} This follows from our observation in {\it Section\,\ref{4-gen}} that $\mathfrak{R}$ is generated by four elements as a $k$-algebra. 

\medskip

{\bf Part (c).} The fundamental pair on $k[\V_4(k)]$ is $(D_4, U_4)$, as defined above. 
Set $\tilde{A}=\krn (D_4)=\bigoplus_{d\in\N}\tilde{A}_d$. Then $\tilde{A}=k[x_0,f,g,q,r]$ where 
$q,r\in A_0$, $x_0,f\in\tilde{A}_4$ and $g\in\tilde{A}_6$; see \cite{Freudenburg.17}, Example 8.12. 
Consequently, $\tilde{A}_d=0$ for odd $d$ and for $d=2$. Since $A_d=\varphi (\tilde{A}_d)$ for all $d\in\N$, it follows that:
\begin{equation}\label{property}
A_d=0 \,\,\text{\it for odd $d$ and for}\,\, d=2\, , \, \text{\it and}\,\, A_0=k.
\end{equation}
According to Panyushev \cite{Panyushev.84}, every algebraic $SL_2(k)$-action on $\A_k^4$ is equal (in some coordinate system) to the action induced by an $SL_2$-module. 
There are four non-trivial $SL_2$-modules of dimension four, namely:
\[
\V_3(k)\,\, ,\,\, \V_1(k)\oplus \V_1(k)\,\, ,\,\, \V_0(k)\oplus \V_2(k)\,\, ,\,\, 2\V_0(k)\oplus \V_1(k)
\]
Assume that $\tilde{X}\subset\A_k^4$ is a subvariety isomorphic to $X$. 
Then there exists an equivariant surjection $\psi : k[W]\to\mathfrak{R}$, where $W$ is one of these four modules. 
Let $(\mathcal{D},\mathcal{U})$ be the fundamental pair on $k[W]$ induced by $W$ and set $\mathcal{A}=\krn\mathcal{D}$ with 
$\N$-grading $\mathcal{A}=\bigoplus_{d\in\N}\mathcal{A}_d$. 

Assume that $W=\V_3(k)$. Then $\mathcal{A}=k[X_0,f,g,h]$ where $h\in\mathcal{A}_0$, $f\in\mathcal{A}_2$, $X_0,g\in\mathcal{A}_3$ and $X_0^2h=f^3+g^2$. 
See \cite{Freudenburg.17}, Example 8.11. 
By (\ref{property}) we have $A_3=\psi (\mathcal{A}_3)=0$, so $f^3=X_0^2h-g^2\in\krn\psi$, which implies $f\in\krn\psi$. Moreover, $\psi (h)\in A_0=k$.
Therefore, $A=\psi (\mathcal{A})\subseteq k$ gives a contradiction. So this case cannot occur.

Assume that $W=\V_1(k)\oplus \V_1(k)$. Then $\mathcal{A}=k[X_0,Y_0,P]$ where $X_0,Y_0\in\mathcal{A}_1$ and $P=X_0Y_1-X_1Y_0\in\mathcal{A}_0$.
By (\ref{property}) we have $A_1=\psi (\mathcal{A}_1)=0$. Since $P\in (X_0,Y_0)$, we see that $A=\psi (\mathcal{A})\subseteq k$, which gives a contradiction. So this case cannot occur.

Assume that $W=2\V_0(k)\oplus \V_1(k)$. Then $\mathcal{A}=k[S, T, X_0]$ where $S,T\in\mathcal{A}_0$ and $X_0\in\mathcal{A}_1$. 
By (\ref{property}) we have $A_1=\psi (\mathcal{A}_1)=0$, so $\psi (X_0)=0$ and $\psi (S),\psi (T)\in k$.
This implies $A=\psi (\mathcal{A})\subseteq k$, a contradiction. So this case cannot occur.

Assume that $W=\V_0(k)\oplus \V_2(k)$. Then $\mathcal{A}=k[S, X_0,f]$ where $S,f\in\mathcal{A}_0$ and $X_0\in\mathcal{A}_2$. 
By (\ref{property}) we have $A_2=\psi (\mathcal{A}_2)=0$.
Therefore, $\psi (X_0)\in A_2=0$, which implies $A=\psi (\mathcal{A})\subseteq k$, a contradiction. So this case cannot occur. 

Therefore, no such mapping $\psi$ exists. This proves part (c).

\medskip

{\bf Part (d).} 
Since $X$ is a homogeneous space for $SL_2(k)$, it follows that $X$ is smooth. In addition, 
\[
k\subseteq ML(\mathfrak{R})\subseteq\krn\delta\cap\krn\upsilon=A_0=k
\]
so $ML(\mathfrak{R})=k$. It was shown in \cite{Freudenburg.24}, Proposition 5.3, that $X$ is not isomorphic to a hypersurface of the form $xy=p(z)$ for $p(z)\in k[z]$. 
{\it Theorem\,\ref{Band-ML}(a)} thus implies that $X$ is not algebraically isomorphic to a hypersurface in $\A_k^3$. 

\medskip

{\bf Part (e).} This follows from ${\rm Cl}(Y\times\A_k^1)={\rm Cl}(Y)\cong\Z$ and ${\rm Cl}(X\times\A_k^1)={\rm Cl}(X)\cong\Z_2$; see
\cite{Fauntleroy.Magid.78}. 

\medskip

This completes the proof of {\it Theorem\,\ref{non-extend}}. 


\section{Proof of {\it Theorem\,\ref{transitive}}}\label{transitive-proof}

\begin{lemma}\label{D-trans} $\mathfrak{D}$ has the strong transitivity property.
\end{lemma}

\begin{proof}  Let $D\in {\rm LND}(\mathfrak{D})$ be irreducible. By {\it Theorem\,\ref{Daigle}}, $\mathfrak{D}$ has the transitivity property. 
{\it Lemma\,\ref{TP-STP}} implies that there exists $\alpha\in {\rm Aut}_k(\mathfrak{D})$ and $c\in k^*$ so that $\alpha D\alpha^{-1}=c\delta'$. 
Define $\gamma\in{\rm Aut}_k(\mathfrak{D})$ by:
\[ 
\gamma (x_0,x_1,x_2)=(cx_0,x_1,c^{-1}x_2)
\]
Then $\gamma\delta'\gamma^{-1}=c\delta'$.  It follows that $D$ is conjugate to $\delta'$. 
\end{proof}

\begin{lemma}\label{derivations} With the hypotheses and notation above:
\begin{itemize} 
\item [{\bf (a)}] Every element of ${\rm Der}_k(\mathfrak{R})$ extends uniquely to an element of ${\rm Der}_k(\mathfrak{D})$. 
\item [{\bf (b)}] Every element of ${\rm LND}(\mathfrak{R})$ extends uniquely to an element of ${\rm LND}(\mathfrak{D})$. 
\item [{\bf (c)}] Every element of ${\rm ILND}(\mathfrak{R})$ extends uniquely to an element of ${\rm ILND}(\mathfrak{D})$. 
\end{itemize}
\end{lemma}

\begin{proof} Given $\theta\in{\rm Der}_k(\mathfrak{R})$ define $\Theta\in {\rm Der}_k(\mathfrak{D})$ by:
\[
\begin{pmatrix} \Theta x_0\cr \Theta x_1\cr \Theta x_2\end{pmatrix} = 
\begin{pmatrix} \textstyle\frac{1}{2}\theta (x_1^2) & -\theta (x_0x_1) & \theta (x_0^2)\cr
\theta (x_1x_2) & -\theta (x_1^2) & \theta (x_0x_1) \cr
\theta (x_2^2) & -\theta (x_1x_2) & \textstyle\frac{1}{2}\theta (x_1^2)\end{pmatrix}
\begin{pmatrix} x_0\cr x_1\cr x_2\end{pmatrix}
\]
The reader can check that $x_2\Theta (x_0)-x_1\Theta (x_1)+x_0\Theta (x_2)=0$. This shows that $\Theta$ is a well-defined derivation of $\mathfrak{D}$. The reader can also check that
\[
\Theta (x_0^2)=\theta (x_0^2) \,\, ,\,\, \Theta (x_0x_1)=\theta (x_0x_1)\,\, ,\,\, \Theta (x_1x_2)=\theta (x_1x_2) \,\, ,\,\, \Theta (x_2^2)=\theta (x_2^2)
\]
which shows that $\Theta$ extends $\theta$. Uniqueness follows from the fact that $\mathfrak{R}\subset \mathfrak{D}$ is an integral extension; see \cite{Freudenburg.17}, Proposition 6.19. This proves part (a). 

For part (b), suppose that $\theta\in {\rm LND}(\mathfrak{R})$. 
It follows from Vasconcelos's Theorem, {\it Theorem\,\ref{Vasconcelos}}, that $\Theta$ is locally nilpotent.

For part (c), suppose that $\theta\in {\rm LND}(\mathfrak{R})$ is irreducible. By {\it Theorem\,\ref{transitivity}} (below), there exists $\alpha\in{\rm Aut}_k(\mathfrak{R})$ such that:
\[
\krn\delta = k[x_0^2]=\alpha (\krn\theta )=\krn (\alpha\theta\alpha^{-1})
\]
Since $\alpha\theta\alpha^{-1}$ is also irreducible, we may assume, with no loss of generality, that $\krn\theta =k[x_0^2]$. 

Let $\Theta$ be the extension of $\theta$ to $\mathfrak{D}$. Then $\krn\Theta =k[x_0]$. Write $\Theta = gD$ for irreducible $D\in {\rm LND}(\mathfrak{D})$ and $g\in k[x_0]$. 
By {\it Lemmma\,\ref{D-trans}}, $\mathfrak{D}$ has the strong transitivity property. Since $\krn D=\krn\delta'$, we see that $D=c\delta'$ for some $c\in k^*$. 
So $\Theta =f\delta'$ for $f=cg$. We have:
\[
fx_0^2=f\delta'(x_0x_1)=\Theta(x_0x_1)=\theta (x_0x_1)\in k[x_0^2] \implies f\in k[x_0^2]
\]
It follows that:
\[
\theta = \Theta\vert_{\mathfrak{R}}=(f\delta')\vert_{\mathfrak{R}}=f(\delta'\vert_{\mathfrak{R}})=f\delta
\]
By irreducibility of $\theta$, $f\in k^*$, and $\Theta$ is irreducible. 
\end{proof}

Let $G=\langle \Delta_f,\tau \, |\, f\in k[x_0]\rangle\subset {\rm Aut}_k(\mathfrak{D})$.  
Daigle's Transitivity Theorem implies that $G$ acts transitively on ${\rm KLND}(\mathfrak{D})$. 
Note that $\tau$ restricts to $\mathfrak{R}$,  and that $\Delta_f$ restricts to $\mathfrak{R}$ if and only if $f\in k[x_0^2]$. 

\begin{theorem}\label{transitivity} Let $\Gamma\subset {\rm Aut}_k(\mathfrak{R})$ be generated by $\{ \Delta_f\, |\, f\in k[x_0^2]\}$ and $\tau\vert_{\mathfrak{R}}$. 
Then $\Gamma$ acts transitively on ${\rm KLND}(\mathfrak{R})$.
\end{theorem}

\begin{proof} Let $\theta\in{\rm LND}(\mathfrak{R})$ be given. Since $\mathfrak{R}\in\mathcal{C}(k)$, $\krn\theta=k[h]=k^{[1]}$ for some $h\in \mathfrak{R}$. 
By {\it Lemma\,\ref{derivations}}, there exists $\Theta\in {\rm LND}(\mathfrak{D})$ which extends $\theta$. 
Consequently, there exists $\alpha\in G$ and $P\in k[x_0]$ such that $\Theta = \alpha (P\delta^{\prime})\alpha^{-1}$
and $\krn\Theta=k[H]$ where $H=\alpha (x_0)$. 

Recall that $\mu\in{\rm Aut}_k(\mathfrak{D})$ is defined by $\mu =(-x_0,-x_1,-x_2)$, $\mu$ restricts to $\mathfrak{R}$ and $\mu\vert_{\mathfrak{R}}=id_{\mathfrak{R}}$. 
Since $^{\mu}h=h$ we have $k[h]\subset k[H]$. Therefore:
\[
k[h]\subset k[H]\cap k[^{\mu}H] \implies \dim ( k[H]\cap k[^{\mu}H])=1\implies k[H]=k[^{\mu}H] \implies ^{\mu}H=\pm H
\]
It follows that $h=H^2=\alpha (x_0)^2$ and $\mu\Theta\mu=\Theta$.

By definition of $G$, there exist $n\in\N$ and $f_1,\hdots ,f_n\in k[x_0]\setminus\{ 0\}$ such that:
\[
\alpha =\tau\Delta_{f_n}\tau\cdots\tau\Delta_{f_1}\tau
\]
Note that $\mu\tau\mu=\tau$ and $\mu\Delta_{f_i}\mu=\Delta_{^{\mu}f_i}$ for each $i$. The equality $\mu\Theta\mu=\Theta$ thus gives:
\[
 \tau\Delta_{f_n}\tau\cdots\tau\Delta_{f_1}\tau (P\delta' ) \tau\Delta_{-f_1}\tau\cdots\tau\Delta_{-f_n}\tau =
\tau\Delta_{^{\mu}f_n}\tau\cdots\tau\Delta_{^{\mu}f_1}\tau (^{\mu}P\delta' ) \tau\Delta_{-^{\mu}f_1}\tau\cdots\tau\Delta_{-^{\mu}f_n}\tau 
\]
Exponentiation gives
\[
 \tau\Delta_{f_n}\tau\cdots\tau\Delta_{f_1}\tau (\Delta_P) \tau\Delta_{-f_1}\tau\cdots\tau\Delta_{-f_n}\tau =
\tau\Delta_{^{\mu}f_n}\tau\cdots\tau\Delta_{^{\mu}f_1}\tau (\Delta_{^{\mu}P}) \tau\Delta_{-^{\mu}f_1}\tau\cdots\tau\Delta_{-^{\mu}f_n}\tau 
\]
which implies:
\[
1=\left( \Delta_{f_n}\tau\cdots\tau\Delta_{f_1}\tau \Delta_{-P} \tau\Delta_{-f_1}\tau\cdots\tau\Delta_{-f_n}\right)
\left( \Delta_{^{\mu}f_n}\tau\cdots\tau\Delta_{^{\mu}f_1}\tau \Delta_{^{\mu}P} \tau\Delta_{-^{\mu}f_1}\tau\cdots\tau\Delta_{-^{\mu}f_n}\right)
\]
According to the amalgamated free product structure of ${\rm Aut}_k(\mathfrak{D})$, 
we can never get 1 by multiplying nonidentity elements chosen alternately from the subgroups $\langle\tau\rangle$ and $\{ \Delta_f\, |\, f\in k[x]\}$;
see \cite{Wright.78}, Proposition 1. 
Therefore, 
\[
1=\Delta_{-f_n}\Delta_{^{\mu}f_n}=\Delta_{-f_n+^{\mu}f_n} \implies ^{\mu}f_n=f_n
\]
and inductively, $^{\mu}P=P$ and $f_i=^{\mu}f_i$ for each $i$. 
Therefore, $P,f_1,\hdots ,f_n\in k[x_0^2]$. This means that $\alpha$ restricts to $\mathfrak{R}$ and $\alpha\in\Gamma$. 
So $k[h]=k[\alpha (x_0^2)]=\alpha (k[x_0^2])$ and the theorem is proved. 
\end{proof}

\begin{corollary}\label{strong} $\Gamma$ acts transitively on ${\rm ILND}(\mathfrak{R})$.
\end{corollary}

\begin{proof} Let $\theta\in {\rm ILND}(\mathfrak{R})$ be given. Since $\Gamma$ acts transitively on ${\rm KLND}(\mathfrak{R})$, by {\it Theorem\,\ref{transitive}}, 
we may, with no loss of generality, assume that $\krn\theta = k[x_0^2]$. 
By {\it Lemma\,\ref{derivations}(c)}, there exists unique $\Theta\in {\rm ILND}(\mathfrak{D})$ extending $\theta$. It follows that $\krn\Theta = k[x_0]$. 
Since $\mathfrak{D}$ has the strong transitivity property, $\Theta =c\delta'$ for some $c\in k^*$, 
and the automorphism $\alpha =(cx_0,x_1,c^{-1}x_2)$ conjugates $\delta'$ to $\Theta$. 
It follows that $\alpha$ restricts to $R$ and conjugates $\delta$ to $\theta$. 
\end{proof}

{\it Corollary\,\ref{strong}} means that $\mathfrak{R}$ has the strong transitivity property. Since $\delta\in {\rm ILND}(\mathfrak{R})$, $\krn\delta =k[x_0^2]$  and ${\rm pl}(\delta )=(x_0^2)$ we
see that the plinth invariant of $\mathfrak{R}$ equals $k[x_0^2]/(x_0^2)\cong_kk$.
This completes the proof of {\it Theorem\,\ref{transitive}}. 


\section{Proof of {\it Theorem\,\ref{aut-group}}}\label{aut-group-proof}

Define the {\bf triangular subgroup} $\mathcal{T}\subset {\rm Aut}_k(\mathfrak{R})$ to be the group of automorphisms preserving $\deg_{\delta}$.
Equivalently, $\alpha\in\mathcal{T}$ if and only if $\alpha (\krn (\delta^n))=\krn (\delta^n)$ for each $n\in\N$. 
\begin{proposition}\label{generators}
Let $T\cong k^*$ be the torus in $PSL_2(k)$ and define:
\[
E=\{ \exp (f\delta )\, |\, f\in k[x_0^2]\}\cong (k^{[1]},+)
\]
\begin{itemize}
\item [{\bf (a)}]  $\mathcal{T}=E\rtimes T$
\item [{\bf (b)}]  ${\rm Aut}_k(\mathfrak{R})$ is generated by $\mathcal{T}$ and $\tau\vert_{\mathfrak{R}}$. 
\end{itemize}
\end{proposition}

\begin{proof}
Part (a). Let $\alpha\in\mathcal{T}$ be given. Since $\alpha$ is triangular, we have
\[
\alpha (x_0^2)=a_0x_0^2+b\,\, ,\,\, \alpha (x_0x_1)=a_1x_0x_1+b_0\,\, ,\,\, \alpha (x_1^2)=a_2x_1^2+b_1 
\]
where $a_i\in k^*$, $b\in k$, $b_0\in k[x_0^2]$ and $b_1\in k[x_0^2]+k[x_0^2]x_0x_1$. By composing with an element of $T$ we may assume that $a_0=1$. 
The relation $(x_0^2)(x_1^2)=(x_0x_1)^2$ yields:
\[
(x_0^2+b)(a_2x_1^2+b_1)-(a_1x_0x_1+b_0)^2=0
\]
The left side of this equation is in $k[x_0^2,x_1]\cong k^{[2]}$, so the coefficient of $x_1^2$ must be 0. We thus obtain:
\[
(a_2-a_1^2)x_0^2+ba_2=0 \,\,\text{in}\,\, k[x_0^2]\implies a_2-a_1^2=ba_2=0 \implies b=0 \implies \alpha (x_0^2)=x_0^2
\]
This leaves the equation $b_1x_0^2=(2a_1+b_0)b_0$, which shows that $\deg_{\delta}(b_1)=0$ and $b_1\in k[x_0^2]$. 
Therefore, $b_1x_0^2=(2a_1+b_0)b_0$ is an equation in $k[x_0^2]$. Since $a_1\in k^*$ we conclude that $b_0\in x_0^2k[x_0^2]$. 

Write $b_0=x_0^2\tilde{b}_0$ for $\tilde{b}_0\in k[x_0^2]$. Using that $\alpha (2x_0x_2-x_1^2)=1$ gives:
\begin{eqnarray*}
x_0^2 &=& x_0^2\,\alpha (2x_0x_2-x_1^2) \\
&=& 2x_0^2\alpha (x_0x_2)-\alpha (x_0x_1)^2 \\
&=& 2x_0^2\alpha (x_0x_2)-(a_1x_0x_1+b_0)^2 \\
&=& 2x_0^2\alpha (x_0x_2)-a_1^2x_0^2x_1^2-2a_1b_0x_0x_1-b_0^2
\end{eqnarray*}
Dividing by $x_0^2$ gives
\begin{eqnarray*}
 1 &=& 2\alpha (x_0x_2)-a_1^2x_1^2-2a_1\tilde{b}_0x_0x_1-x_0^2\tilde{b}_0^2 \\
 &=& 2\alpha (x_0x_2)-a_1^2(2x_0x_2-1)-2a_1\tilde{b}_0x_0x_1-x_0^2\tilde{b}_0^2 
\end{eqnarray*}
which shows that $2\alpha (x_0x_2)$ has the form $x_0P+ (1-a_1^2)$ for some $P\in \mathfrak{D}$. Squaring thus gives
$4x_0^2\alpha (x_2^2)=x_0Q+(1-a_1^2)^2$ for some $Q\in\mathfrak{D}$, so $(1-a_1^2)^2\in x_0\mathfrak{D}$. Since $a_1\in k^*$ this means $1-a_1^2=0$. 
Let $\rho\in T$ be induced by $(-x_0,x_1,-x_2)\in {\rm Aut}_k(\mathfrak{D})$. Then either:
\[
\alpha (x_0^2)=x_0^2 \,\,\text{and}\,\, \alpha (x_0x_1)=x_0x_1+b_0 \quad \text{or}\quad 
\rho\alpha (x_0^2)=x_0^2 \,\,\text{and}\,\, \rho\alpha (x_0x_1)=x_0x_1+\rho (b_0)
\]
We can thus assume, with no loss of generality, that $\alpha (x_0^2)=x_0^2$ and $\alpha (x_0x_1)=x_0x_1+b_0$. 

Let $S=\{ (x_0^2)^m\, |\, m\in\N\}$ and define:
\[
L=S^{-1}k[x_0^2] \quad \text{and}\quad \mathfrak{R}'=S^{-1}\mathfrak{R}
\]
Then $\delta$ and $\alpha$ extend to $\epsilon\in {\rm LND}(\mathfrak{R}')$ and $\beta\in {\rm Aut}_L(\mathfrak{R}')$, respectively. 
Since $\epsilon$ has slice $(x_0^2)^{-1}x_0x_1$, the Slice Theorem implies $\mathfrak{R}'=L[x_0x_1]=L^{[1]}$. 
Since $\beta (x_0x_1)=x_0x_1+b_0$ we can write $\beta =\exp (\Delta)$ for some $\Delta\in {\rm LND}(\mathfrak{R}')$.
Therefore, $\Delta\vert_{\mathfrak{R}}\in {\rm LND}(\mathfrak{R})$ and $\alpha=\exp (\Delta\vert_{\mathfrak{R}})$. In addition, $\krn\Delta\vert_{\mathfrak{R}}=\mathfrak{R}\cap L=k[x_0^2]$,
    which implies $\Delta\vert_{\mathfrak{R}}=f\delta$ for some $f\in k[x_0^2]$. So $\alpha=\exp (f\delta )$. This proves part (a).

Part (b). Let $\sigma\in {\rm Aut}_k(\mathfrak{R})$ be given. 
By {\it Theorem\,\ref{transitivity}}, there exists $\gamma\in\Gamma=\langle E,\tau\vert_{\mathfrak{R}}\rangle$ such that $\sigma\delta\sigma^{-1}=\gamma^{-1}\delta\gamma$. We thus have, for each $n\in\N$,
\[
\delta^n=(\gamma\sigma)^{-1}\delta^n (\gamma\sigma) \implies \krn \delta^n=\gamma\sigma (\krn \delta^n)\implies \gamma\sigma\in\mathcal{T}
\]
    It follows that ${\rm Aut}_k(\mathfrak{R})$ is generated by $\mathcal{T}$ and $\tau\vert_{\mathfrak{R}}$. 
\end{proof}

Given $\varepsilon\in E$, let $\varepsilon =\exp (f\delta)$ for $f\in k[x_0^2]$. Then $\varepsilon$ is the restriction of $\exp (f\delta )\in {\rm Aut}_k(\mathfrak{D})$. 
Similarly, every $\lambda\in k^*\subset {\rm Aut}_k(\mathfrak{R})$ is the restriction of $\Lambda\in k^*\subset {\rm Aut}_k(\mathfrak{D})$. 
Therefore, part (a) of {\it Proposition\,\ref{generators}} shows that every $\alpha\in\mathcal{T}$ is the restriction of some triangular element of ${\rm Aut}_k(\mathfrak{D})$,
and part (b) shows that every $\alpha\in {\rm Aut}_k(\mathfrak{R})$ is the restriction of some element of ${\rm Aut}_k(\mathfrak{D})$.
This proves part (a) of {\it Theorem\,\ref{aut-group}}. 

Part (b) of  {\it Theorem\,\ref{aut-group}} now follows from {\it Theorem\,\ref{Wright}}.
 

\section{Proof of {\it Theorem\,\ref{non-cancel}}}\label{non-cancel-proof}

We have $\mathfrak{R}=k[x,y,z,w]$ with fundamental pair $(\delta ,\upsilon )$ where:
\[
x=x_0^2\,\, ,\,\, y=x_0x_1\,\, ,\,\, z=x_1x_2\,\, ,\,\,w=x_2^2
\]
Define $P=x_0x_2=2xw-yz$. From {\it Section\,\ref{4-gen}} we obtain relations in $\mathfrak{R}$ 
\[
y^2=x(2P-1)\,\, ,\,\, yP=xz\,\, ,\,\, yz=P(2P-1)\,\, ,\,\, yw=zP\,\, ,\,\, z^2=w(2P-1)
\]
and $\delta$ operates on $\mathfrak{R}$ by:
\[
\delta w=2z\,\, ,\,\, \delta z=3P-1\,\, ,\,\, \delta P=y\,\, ,\,\, \delta y=x\,\, ,\,\, \delta x=0
\]
Note that $\delta (zP)=5P^2-2P$. 

Extend $\delta$ to $\Delta$ on $\mathfrak{R}[v]=\mathfrak{R}^{[1]}$ by $\Delta v=w$. 
Define $Q=5xv-yw$. Then $\Delta Q=2P$, 
so $\Delta s=1$ for $s=\frac{3}{2}Q -z$. Set $\tilde{R}=\krn\Delta$. By the Slice Theorem,
\[
\mathfrak{R}^{[1]}=\mathfrak{R}[v]=k[x,y,z,w,v]=k[x,y,s,w,v]=\tilde{\mathfrak{R}}[s]=\tilde{\mathfrak{R}}^{[1]}
\]
and
\[
\tilde{\mathfrak{R}}\cong_k\mathfrak{R}[v]/(s)=k[\bar{x},\bar{y},\bar{w},\bar{v}]
\]
where $2\bar{z}=3\bar{Q}$ and $2\bar{P}=4\bar{x}\bar{w}-3\bar{y}\bar{Q}$. 
This proves part (a) of {\it Theorem\,\ref{non-cancel}}.

Let $L=k[\bar{x},\bar{x}^{-1}]$ and observe that $\tilde{\mathfrak{R}}[\bar{x}^{-1}]=L[\bar{y},\bar{w},\bar{v}]=L[\bar{y},\bar{P},\bar{Q}]$. 
Since $3\bar{x}\bar{Q}=2\bar{y}\bar{P}$ and $\bar{x}(2\bar{P}-1)=\bar{y}^2$ we see that $\tilde{\mathfrak{R}}[\bar{x}^{-1}]=L[\bar{y}]=L^{[1]}$. Define $\Theta\in{\rm LND}(L[\bar{y}])$ 
by $\Theta \bar{y}=\bar{x}$. From the observed relations we find:
\[ \textstyle
\Theta \bar{P}=\bar{y}\,\, ,\,\, \Theta \bar{Q}=2\bar{P}-\frac{2}{3} \,\, ,\,\, \Theta \bar{w}= 3\bar{Q} \,\, ,\,\, \Theta \bar{v}=\bar{w}-\frac{2}{15}\bar{x}^{-1}
\]
Define $\theta=\bar{x}\,\Theta\in{\rm LND}(\tilde{\mathfrak{R}})$. Let $J=(\theta\tilde{\mathfrak{R}})$, the ideal generated by the image of $\theta$. Then
\[
J=(\theta \bar{x},\theta \bar{y},\theta \bar{w},\theta \bar{v})=(\bar{x}^2,\bar{x}\bar{Q},\bar{x}\bar{w}-\textstyle\frac{2}{15})
\]
Modulo $J$, we have:
\[
0\equiv \bar{x}\bar{Q}\equiv 5\bar{x}^2\bar{v}-\bar{x}\bar{y}\bar{w}\equiv\textstyle\frac{2}{15}\bar{y} \implies
0\equiv 2\bar{P}-\textstyle\frac{2}{3}\equiv 4\bar{x}\bar{w}-2\bar{y}\bar{z}-\frac{2}{3}\equiv -\frac{2}{15} \implies J=(1)
\]
Therefore, $\theta$ is irreducible and the $\G_a$-action on $\tilde{X}$ induced by $\theta$ is fixed-point free. 
We thus obtain:
\begin{enumerate}
\item $\theta\in{\rm LND}(\tilde{\mathfrak{R}})$ is irreducible.
\item $\krn\theta =k[\bar{x}]$
\item $\theta \bar{y}=\bar{x}^2$ and ${\rm pl}(\theta )=\bar{x}^2k[\bar{x}]$
\end{enumerate}
Since the plinth invariant of $\mathfrak{R}$ equals $k$, by {\it Theorem\,\ref{transitive}}, there is no $D\in{\rm ILND}(\mathfrak{R})$ such that $\krn D/{\rm pl}(D)\cong k[x]/(x^2)$. Therefore,
$\tilde{\mathfrak{R}}\not\cong_k\mathfrak{R}$.

In addition, $\tilde{\mathfrak{R}}$ is $\Z$-graded with $\bar{x},\bar{y},\bar{w},\bar{v}$ homogeneous and $\deg (\bar{x},\bar{y},\bar{w},\bar{v})=(2,1,-2,-3)$. 
The derivation $\theta$ is homogeneous and $\deg\theta =3$.
This grading induces a $\G_m$-action on $\tilde{X}$ which semi-commutes with the $\G_a$-action. 
This proves parts (b) and (c) of {\it Theorem\,\ref{non-cancel}}.

For part (d), $\tilde{\mathfrak{R}}\in\mathcal{C}(k)$ follows from part(a) and {\it Lemma\,\ref{cylinders}} below. 
Part (a) also shows that $\mathfrak{R}$ and $\tilde{\mathfrak{R}}$ have the same class group. 

This completes the proof of {\it Theorem\,\ref{non-cancel}}.


\section{Proof of {\it Theorem\,\ref{plinth-invariant}}}\label{plinth-invariant-proof}

By {\it Lemma\,\ref{D-trans}}, $B_1=\mathfrak{D}$ has the strong transitivity property. Since $\delta'\in {\rm ILND}(B_1)$ we see that the plinth invariant of 
$B_1$ is $k[x]/(x)=k$. 

Consider $B_n$ for $n\ge 2$, where $x^nz-y^2-1=0$. It was shown by Makar-Limanov \cite{Makar-Limanov.01a} that $ML(B_n)=k[x]\cong k^{[1]}$. 
So each $D\in {\rm LND}(B)$ has $\krn D=k[x]$, which implies $B_n\in\mathcal{C}_1(k)$ and $B_n$ has the transitivity property. 
The proof that $B_n$ has the strong transitivity property is almost identical to the proof of {\it Lemma\,\ref{D-trans}}, and is left to the reader. 

Define $\omega\in {\rm ILND}(B_n)$ by $\omega =x^n\partial_y+2y\partial_z$. Then $\krn\omega =k[x]$ and ${\rm pl}(\omega )=(x^n)$, which 
implies that the plinth invariant of $B_n$ is $k[x]/(x^n)$. 

This proves {\it Theorem\,\ref{plinth-invariant}}. 



\section{Three lemmas for the class $\mathcal{C}(k)$}

Recall the conditions for membership in $\mathcal{C}(k)$.
\begin{itemize}
\item [(a)] $B$ is a normal affine $k$-domain, $\dim B=2$, $B^*=k^*$ and $ML(B)\ne B$.
\item [(b)] $\krn D\cong_kk^{[1]}$ for every nonzero $D\in {\rm LND}(B)$. 
\end{itemize}

\begin{lemma}\label{C(k)} Suppose that the ring $B$ satisfies condition {\rm (a)}. If one of the conditions
\begin{enumerate}[label=(\roman*)]
\item $k=\bar{k}$ and $B$ is rational.
\item $k=\bar{k}$ and $B$ is a UFD.
\item $ML(B)=k$
\end{enumerate}
holds, then $B\in\mathcal{C}(k)$. 
\end{lemma}

\begin{proof} Let $D\in {\rm LND}(B)$, $D\ne 0$, and set $A=\krn D$. 

Assume that (i) holds. The existence of a local slice for $D$ implies ${\rm frac}(B)\cong {\rm frac}(A)^{(1)}$. 
By hypothesis, ${\rm frac}(B)\cong k^{(2)}$, so by L\"uroth's Theorem we see that ${\rm frac}(A)\cong k^{(1)}$. In addition, $A$ is normal (since $B$ is normal) and $A^*=k^*$. 
Therefore, ${\rm Spec}(A)$ is a smooth rational affine curve with trivial units, i.e., ${\rm Spec}(A)\cong\A_k^1$. 

Assume that condition (ii) holds. Then $A$ is a UFD; see \cite{Freudenburg.17}, Lemma 2.8. Since $\dim A=1$ and $A^*=k^*$, it follows that $A\cong k^{[1]}$; see
\cite{Freudenburg.17}, Lemma 2.9. 

Assume that condition (iii) holds. Then $A\cong k^{[1]}$; see Lemma 2.14 of \cite{Kolhatkar.11}.
\end{proof} 

\begin{lemma}\label{cylinders} Assume that $k=\bar{k}$ and let $B_1\in\mathcal{C}(k)$. If $B_2$ is a $k$-domain such that $B_1^{[1]}\cong_kB_2^{[1]}$ then $B_2\in\mathcal{C}(k)$.
\end{lemma}

\begin{proof} The hypotheses imply that $B_2$ is a normal affine $k$-domain, $\dim B_2=2$, and $B_2^*=k^*$. If $ML(B_2)=B_2$ then $ML(B_2^{[1]})=B_2$  (see \cite{Freudenburg.17}, Theorem 2.24).
We thus have
\[
B_2=ML(B_2^{[1]})=ML(B_1^{[1]})\subseteq ML(B_1) \implies 2=\dim B_2\le \dim ML(B_1)\le 1
\]
which is a contradiction. Therefore, $ML(B_2)\ne B_2$ and $B_2$ satisfies condition (a). 

For condition (b), let $D\in {\rm LND}(B_2)$ be given, $D\ne 0$. 
Since $B_1$ is rational, $B_1^{[1]}=B_2^{[1]}$ is rational. By L\"uroth's Theorem we see that $\krn D$ is rational. 
So ${\rm Spec}(\krn D)$ is a normal (smooth) rational affine curve with trivial units, which means ${\rm Spec}(\krn D)\cong\A_k^1$ and $\krn D\cong_kk^{[1]}$. 
\end{proof}

\begin{lemma}\label{TP-STP} Suppose that $B\in\mathcal{C}(k)$ has the transitivity property. 
\begin{itemize} 
\item [{\bf (a)}] Given $D,E\in {\rm ILND}(B)$, there exists $\alpha\in{\rm Aut}_k(B)$ and $c\in k^*$ such that $\alpha E\alpha^{-1}=cD$. 
\item [{\bf (b)}] If $k=\bar{k}$ and $B$ admits a nontrivial $\Z$-grading then has the strong transitivity property. 
\end{itemize}
\end{lemma} 

\begin{proof} Part (a). Set $A=\krn D\cong k^{[1]}$. By hypothesis, there exists $\alpha\in{\rm Aut}_k(B)$ such that $\krn (\alpha E\alpha^{-1})=A$. 
We may thus assume, with no loss of generality, that $\krn E=A$. By \cite{Freudenburg.17}, Principle 12, there exist $a,b\in A$ such that $aD=bE$. We may assume that $\gcd_A(a,b)=1$.
Let $f,g\in A$ be such that $af+bg=1$. We have:
\[
gaD=gbE=(1-af)E \implies E=a(gD-fE)
\]
Since $E$ is irreducible as a derivation it follows that $a\in k^*$. Irreducibility of $D$ now implies $b\in k^*$. So $E=cD$ for some $c\in k^*$. 

Part (b). Let $B=\bigoplus_{d\in\Z}B_d$ be a nontrivial $\Z$-grading and let $\G_m\to {\rm Aut}_k(B)$, $t\to \lambda_t$, be the induced
$\G_m$-action. There exists homogeneous $D\in {\rm ILND}(B)$; see \cite{Freudenburg.17}, Principle 15 and Theorem 2.15. 
If $d=\deg D$ then $\lambda_tD\lambda_t^{-1}=t^dD$ for all $t\in k^*$. Since $k=\bar{k}$ we see that $D$ and $cD$ are conjugate for all $c\in k^*$. 
Combined with part (a), this proves part (b). 
\end{proof}


\section{Additional notes and Questions}\label{last}

\subsection{Extensions of group actions}

Assume that $k=\C$. {\it Theorem\,\ref{non-extend}} shows that the $SL_2(\C )$-action on $X$ does not extend to $\A_{\C }^4$. 
Examples of non-extension of a group action were previously known, but this appears to be the first for $SL_2(\C )$. 
Following is an example for the torus $\C^*$. 

Suppose that $n\in\N$, $p(y)\in\C [y]$, $\deg p(y)\ge 2$, and $r(x)\in \C [x]$, $\deg r(x)\le n-1$ and $r(0)=1$. Define the sequence $d_m\in\N$ by $d_0=d_1=1$ and $d_{m+1}=3d_m-d_{m-1}$. 
Define hypersurfaces in $U,U',,V_n,V_n^{\prime},W_m\subset \A_{\C}^3$ by:
\begin{eqnarray*}
U&:& x^2z = y^2-1 \\ U'&:& x^2z = (1+x)y^2-1 \\ V_n&:& x^nz = p(y) \\ V_n^{\prime}&:& x^nz = r(x)p(y) \\ W_m&:& xz = y^{d_m}-1
\end{eqnarray*}
In \cite{Freudenburg.Moser-Jauslin.03} the author and Moser-Jauslin show the following.
\begin{enumerate}
\item $U$ and $U'$ are algebraically non-isomorphic but there is a holomorphic automorphism of $\C^3$ which carries $U$ to $U'$.
\item For each $n\ge 2$, $V_n\cong _kV_n^{\prime}$ but there is no algebraic automorphism of $\A_{\C}^3$ which carries $V_n$ to $V_n^{\prime}$. 
\item In case $p(y)=y^2-1$ and $r(x)=1+x$ there is, for each $n\ge 2$, a holomorphic automorphism of $\C^3$ which carries $V_n$ to $V_n^{\prime}$. 
\item For each $m\ge 3$, there exists a hypersurface $W_m^{\prime}\subset\A_{\C}^3$ such that $W_m\cong _kW_m^{\prime}$ 
but there is no algebraic automorphism of $\A_{\C}^3$ which carries $W_m$ to $W_m^{\prime}$. 
\end{enumerate}
Given $n\ge 2$, the $\C^*$-action on $\A_{\C}^3$ defined by $t\cdot (x,y,z)=(tx,y,t^{-n}z)$ restricts to a nontrivial $\C^*$ action on $V_n$. 
Therefore, $V_n^{\prime}$ has a nontrivial algebraic $\C^*$-action. 
In \cite{Dubouloz.Poloni.09}, Dubouloz and Poloni show that, when $r(x)=1-x$ and $p(y)$ has simple roots, the $\C^*$-action on $V_n^{\prime}$ 
does not extend algebraically to $\A_k^3$, but does extend holomorphically to $\C^3$. 
So the $\C^*$-action extends to $\A_{\C}^3$ for one embedding $V_n$, but not for the other embedding $V_n^{\prime}$. 
In contrast, {\it Theorem\,\ref{non-extend}} shows that the $SL_2(\C )$-action on $X$ fails to extend to $\A_{\C}^4$ for {\it all} embeddings of $X$ in $\A_{\C}^4$. 

We ask the following.
\begin{question} 
Are all algebraic embeddings of $Y$ in $\A_{\C}^3$ equivalent?
\end{question}

\begin{question} 
Are all algebraic embeddings of $X$ in $\A_{\C}^4$ equivalent?
\end{question}

\begin{question} 
Can the $SL_2(\C )$-action on $X$ be extended holomorphically to $\C^4$? 
\end{question}

\subsection{The exponential closure of a group of automorphisms}
The automorphism groups for $k^{[2]}$, $\mathfrak{D}$ and $\mathfrak{R}$ are structurally quite similar. A closer look at these similarities motivates the following definition. 
Assume that $B$ is an affine $k$-domain. 

\begin{definition} Let $H\subset{\rm Aut}_k(B)$ be a subgroup. 
The {\bf exponential closure} of $H$ is the subgroup generated by $H$ and the set of exponential automorphisms:
\[
\{ \exp (fD)\in{\rm Aut}_k(B)\, |\, D\in {\rm LND}(B)\, ,\, \exp (D)\in H\, ,\, f\in{\rm frac}(\krn D)\}
\]
\end{definition}
For each of the the rings $k^{[2]}$, $\mathfrak{D}$ and $\mathfrak{R}$, its automorphism group is the exponential closure of its linear subgroup:
${\rm Aut}_k(k^{[2]})$ is the exponential closure of $GL_2(k)$,
${\rm Aut}_k(\mathfrak{D})$ is the exponential closure of $O_3(k)$, and 
${\rm Aut}_k(\mathfrak{R})$ is the exponential closure of $PSL_2(k)$.

Linear $\G_a$-actions on the polynomial ring $k^{[n]}$ are in bijective correspondence with $SL_2$-modules of dimension $n$; see \cite{Freudenburg.17}, Proposition 6.3. 
For $k^{[2]}$, there is only one nontrivial $SL_2$-module, namely, $\V_1(k)$, so all linear $\G_a$-actions on $k^{[2]}$ are conjugate by a linear automorphism. 
For $k[x,y,z]=k^{[3]}$ there are two distinct possibilities, $\V_2(k)$ and $\V_0(k)\oplus\V_1(k)$. 
The associated locally nilpotent derivations are $D_2=x\partial_y+y\partial_z$ and $D_0+D_1=x\partial_y$. 
The exponential closure of $GL_3(k)$ in ${\rm Aut}_k(k^{[3]})$ thus contains every triangular automorphism, and therefore contains the tame subgroup, and also contains 
non-tame automorphisms such as the Nagata automorphism $\exp (fD_2)$, $f=2xz-y^2$. The homogeneous $(2,5)$ derivation $D$ of $k^{[3]}$ is irreducible, and it seems unlikely that $\exp (D)$ belongs
to the exponential closure of $GL_3(k)$; see \cite{Freudenburg.17}, Section 5.4. 



\vspace{.2in}

\noindent \address{Department of Mathematics\\
Western Michigan University\\
1903 W. Michigan Ave.\\
Kalamazoo, Michigan 49008} \,\,USA\\
\email{gene.freudenburg@wmich.edu}

\end{document}